\documentclass[11pt]{article}

\usepackage[T1]{fontenc}
\usepackage{lmodern}
\usepackage{microtype}
\usepackage[a4paper,margin=1in]{geometry}
\usepackage{amsmath,amssymb,amsthm,mathtools}
\usepackage{enumitem}
\usepackage{booktabs}
\usepackage{xcolor}
\usepackage{hyperref}
\usepackage{bookmark}
\usepackage{tikz}
\usepackage{graphicx}
\usepackage{float}
\usetikzlibrary{decorations.pathreplacing}

\hypersetup{
  colorlinks=true,
  linkcolor=blue!45!black,
  citecolor=blue!45!black,
  urlcolor=blue!45!black,
  pdftitle={Bounded Twin-Width Tournaments are Chi-Bounded},
  pdfauthor={Chaoliang Tang and Junchi Zhang}
}

\setlist[itemize]{leftmargin=1.5em,itemsep=0.15em,topsep=0.3em}

\newtheorem{theorem}{Theorem}[section]
\newtheorem{lemma}[theorem]{Lemma}
\newtheorem{corollary}[theorem]{Corollary}

\newcommand{\tww}{\operatorname{tww}}
\newcommand{\dchi}{\vec{\chi}}
\newcommand{\domega}{\vec{\omega}}
\newcommand{\Inv}{\operatorname{Inv}}
\newcommand{\cF}{\mathcal{F}}
\newcommand{\cC}{\mathcal{C}}

\title{Bounded Twin-Width Tournaments are $\dchi$-Bounded}
\author{%
  Chaoliang Tang\thanks{Fudan University. Emails: \texttt{\{cltang22, jczhang24\}@m.fudan.edu.cn}}
  \and
  Junchi Zhang\footnotemark[1]
}
\date{}

\begin{document}
\maketitle
\vspace{-2.0em}

\begin{abstract}
For a tournament $T$ and a vertex ordering $\prec$, let $T^{\prec}$ be the graph of backward arcs in $\prec$. The diclique number of a tournament is $\domega(T)=\min_{\prec}\omega(T^{\prec})$, and the dichromatic number is $\dchi(T)=\min_{\prec}\chi(T^{\prec})$. We prove a mixed parameter transfer theorem: for all $k$ and $r$, if $\tww(T)\le k$ and $\omega(T^{\prec})\le r$, then the ordered twin-width of $(T^{\prec},\prec)$ is bounded by a function of $k$ and $r$. The proof combines the regular-semigrid theorem for ordered graphs with permutation-encoding obstructions to bounded twin-width in tournaments. Together with polynomial $\chi$-boundedness of graphs of bounded twin-width, this implies that tournaments of bounded twin-width are $\dchi$-bounded by $\domega$, resolving a conjecture of Aboulker, Aubian, Charbit, and Lopes.
\end{abstract}

\section{Introduction}
Coloring directed graphs naturally leads to the dichromatic number, introduced by Neumann-Lara~\cite{NeumannLara}, where independent sets are replaced by acyclic sets. For tournaments, acyclic sets are exactly transitive subtournaments, and this viewpoint has led to a rich coloring theory, including the characterization of heroes by Berger et al.~\cite{BCCHFLSST}. Another useful connection with undirected graphs comes from backward graphs: a vertex ordering together with its backward graph completely determines the tournament, while suitable orderings translate questions on dichromatic number into ordinary graph coloring. This perspective has also been used in structural problems on tournaments, for instance in work on the strong Erd\H{o}s--Hajnal property~\cite{CSSS} and on clique-type parameters defined via backward graphs~\cite{AACL}. It therefore provides a natural bridge between directed and undirected graph theory, which is also central to the questions considered in this paper.

A \emph{tournament} is an orientation of a complete graph. For a total order $\prec$ of $V(T)$, the \emph{backward graph} $T^{\prec}$ is the graph on $V(T)$ such that, for $x\prec y$,
\[
xy\in E(T^{\prec}) \quad\Longleftrightarrow\quad y\to x\text{ in }T.
\]
The \emph{dichromatic number} $\dchi(T)$ is the minimum number of parts in a partition of $V(T)$ into acyclic subtournaments; such a partition is called a \emph{dicoloring}. Equivalently, $\dchi(T)$ is the minimum chromatic number of $T^{\prec}$ over all total orders $\prec$ of $V(T)$. We also introduce the term \emph{diclique number} for $\domega(T)$, defined by the second fomula below: 
\begin{equation}\label{eq:parameters}
\dchi(T)=\min_{\prec}\chi(T^{\prec}),
\qquad \domega(T):=\min_{\prec}\omega(T^{\prec}).
\end{equation}
A tournament class $\mathcal T$ is $\dchi$-bounded if there is a function $f\colon\mathbb N\to\mathbb N$ such that $\dchi(T)\le f(\domega(T))$ for every $T\in\mathcal T$.

Twin-width was introduced by Bonnet, Kim, Thomass\'e, and Watrigant~\cite{BKTW}. We use its standard formulation for finite binary relational structures. Let $\mathcal A$ be such a structure and put $n:=|V(\mathcal A)|$. A \emph{contraction sequence} is a sequence $\mathcal P_n,\ldots,\mathcal P_1$ of partitions of $V(\mathcal A)$, starting with the singleton partition and ending with $\mathcal P_1=\{V(\mathcal A)\}$, in which $\mathcal P_{i-1}$ is obtained from $\mathcal P_i$ by merging two parts. At a partition, distinct parts $A,B$ are in \emph{conflict} if some basic binary relation of $\mathcal A$ is nonconstant on $A\times B$ or on $B\times A$. The width of the sequence is the largest number of parts in conflict with any one part at any stage, and $\tww(\mathcal A)$ is the minimum width of a contraction sequence.

An \emph{ordered graph} $(G,\prec)$ is a graph equipped with a total order of its vertices. Thus $\tww(G)$ uses the adjacency relation of a graph, $\tww(T)$ uses the arc relation of a tournament, and $\tww(G,\prec)$ uses both adjacency and the order relation. In particular,
\[
\tww(G)\le \tww(G,\prec),
\]
since forgetting the order cannot create conflicts.

The parameter $\domega$ already appeared in Ilhee Kim's doctoral thesis~\cite{Kim}, and clique numbers of backward graphs of tournaments were also considered by Nguyen, Scott, and Seymour~\cite{NSS}. Its systematic study was initiated by Aboulker, Aubian, Charbit, and Lopes~\cite{AACL}, who developed the $\dchi$-boundedness framework for tournaments. They focused on the relationship between $\domega(T)$ and $\dchi(T)$, studied forbidden-subtournament classes, and proved that substitution preserves $\dchi$-boundedness. In particular, they asked whether every class of tournaments of bounded twin-width is $\dchi$-bounded by $\domega$ (see~\cite[Conjecture~3.13]{AACL}).

The following transfer theorem is our main result. 

\begin{theorem}\label{thm:transfer}
For every pair of integers $k,r\ge 0$ there exists $h(k,r)$ such that, for every tournament $T$ and every order $\prec$,
\begin{equation}\label{eq:transfer}
\tww(T)\le k,\quad \omega(T^{\prec})\le r
\quad\Longrightarrow\quad
\tww(T^{\prec},\prec)\le h(k,r).
\end{equation}
\end{theorem}

As an immediate consequence, we answer Conjecture~3.13 of~\cite{AACL} affirmatively.

\begin{corollary}\label{cor:main}
For every integer $k\ge 0$ there exists a function $f_k$ such that every tournament $T$ with $\tww(T)\le k$ satisfies
\[
\dchi(T)\le f_k\bigl(\domega(T)\bigr).
\]
\end{corollary}

The $\chi$-boundedness of graphs of bounded twin-width was first proved by Bonnet, Geniet, Kim, Thomass\'e, and Watrigant~\cite{BGKTW}. A quasi-polynomial bound was subsequently obtained by Pilipczuk and Soko{\l}owski~\cite{PS}, followed by the polynomial bound below.

\begin{lemma}[Bourneuf--Thomass\'e, Theorem~1.2 in~\cite{BT}]\label{lem:chi}
For every $d\ge0$, there is a polynomial $p_d$ such that every graph $H$ with $\tww(H)\le d$ satisfies
\[
\chi(H)\le p_d(\omega(H)).
\]
\end{lemma}

The dependence on $\domega(T)$ is necessary: Aboulker, Aubian, Charbit, and Lopes~\cite{AACL} construct tournaments of twin-width $1$ and unbounded dichromatic number. Lemma~\ref{lem:chi} suggests applying polynomial $\chi$-boundedness to a backward graph. A \emph{binary-search-tree (BST) order} is the left-to-right order of a binary tree on $V(T)$ in which the left and right subtrees of each vertex contain only its in- and out-neighbors, respectively. Geniet and Thomass\'e~\cite[Lemma~1.4]{GT} proved that every BST order has ordered twin-width bounded in terms of the tournament twin-width. The remaining ordering gap is that a low-width BST order need not minimize the clique number of the backward graph. Conjecture~3.17 of~\cite{AACL} would close this gap by asserting that some BST order has backward-graph clique number bounded in terms of $\domega(T)$. Here we close the gap by the different route expressed in Theorem~\ref{thm:transfer}. Applying~\eqref{eq:transfer} to an order attaining $\domega(T)$, followed by Lemma~\ref{lem:chi}, yields Corollary~\ref{cor:main}.\\

To prove Theorem~\ref{thm:transfer}, we need some known facts about semigrids and permutation-encoding tournaments. Related structural and coloring questions for ordered graphs have also been studied through complete interval minors~\cite{BCTT}.

For disjoint vertex sets $A,B$ in an ordered graph, write $A\prec B$ if $a\prec b$ for every $a\in A$ and $b\in B$; a $\prec$-interval is a set of consecutive vertices. A regular $n\times n$
semigrid, as introduced by Simon and Toru\'nczyk~\cite{ST}, is an ordered graph whose vertex set is partitioned into $\prec$-intervals
\[
I_0,I_1,\ldots,I_n,\quad |I_i|=n+1
\]
such that $I_1\prec\cdots\prec I_n$ and $I_0$ lies entirely before or after their union. For $0\le i\le n$, write
\[
I_i=v_{i,1}\prec\cdots\prec v_{i,n+1}.
\]
The scheme of a regular semigrid records the following data. First, it
specifies a relation $R\in\{=,\neq,\le,\ge\}$ governing the adjacencies
between $I_0$ and $U:=I_1\cup\cdots\cup I_n$, for $1\le s,t\le n+1$ and $1\le j\le n$:
\begin{equation}\label{eq:semigrid-cross}
v_{0,s} v_{j,t}\in E(G) \quad\Longleftrightarrow\quad s\,R\,t.
\end{equation}
It also specifies whether $I_0$ is a clique or an independent set,
whether $I_0\prec U$ or $U\prec I_0$, and the edge/non-edge value for
each of the four relative-position types within $U$: vertices in the
same row, vertices in the same column, and vertices in different rows
whose column indices are respectively increasing or decreasing.
Consequently, there are
\[
4\cdot 2\cdot 2\cdot 2^4=256
\]
possible regular-semigrid schemes~\cite[Section~4]{ST}.

A class of ordered graphs or tournaments is \emph{hereditary} if it is closed under induced subgraphs or subtournaments, with the order restricted in the ordered case.

The relevance of these configurations comes from the grid theorem of Simon and Toru\'nczyk: unbounded twin-width in a hereditary class of ordered graphs forces regular square semigrids of arbitrarily large size.

\begin{lemma}[Simon--Toru\'nczyk, Theorem~1.1 in~\cite{ST}]\label{lem:semigrid}
Every hereditary class of ordered graphs of unbounded twin-width contains arbitrarily large regular square semigrids.
\end{lemma}

For an ordered graph $(G,\prec)$, let $\Inv(G,\prec)$ be the tournament obtained, for $x\prec y$, by
\begin{equation}\label{eq:inverse}
y\to x \Longleftrightarrow xy\in E(G),
\qquad x\to y \Longleftrightarrow xy\notin E(G).
\end{equation}
Thus $(\Inv(G,\prec))^{\prec}=G$. For $n\ge1$, let $[n]:=\{1,\ldots,n\}$ and $\mathfrak S_n$ denote the permutations of $[n]$. For $Q\in\{=,\le,\ge\}$ and $\pi\in\mathfrak S_n$, let
$F_{Q,n}(\pi)$ be the tournament on vertices $x_1,\ldots,x_n,y_1,\ldots,y_n$ whose arcs are defined, for $1\le i,j\le n$, by
\begin{equation}\label{eq:permutation-encoding}
\begin{aligned}
x_i\to x_j \Longleftrightarrow i<j, \qquad
y_i\to y_j \Longleftrightarrow i<j,\qquad
x_i\to y_j \Longleftrightarrow i\,Q\,\pi^{-1}(j).
\end{aligned}
\end{equation}

\begin{center}
\resizebox{0.82\textwidth}{!}{%
\begin{tikzpicture}[
    vertex/.style={circle,fill=black,inner sep=2.2pt},
    >=stealth,
    font=\scriptsize
]

\begin{scope}[xshift=0cm]
\node[vertex,label=left:$x_1$] (x1) at (0,0) {};
\node[vertex,label=left:$x_2$] (x2) at (0,1) {};
\node[vertex,label=left:$x_3$] (x3) at (0,2) {};
\node[vertex,label=left:$x_4$] (x4) at (0,3) {};

\node[vertex,label=right:$y_1$] (y1) at (1.8,0) {};
\node[vertex,label=right:$y_2$] (y2) at (1.8,1) {};
\node[vertex,label=right:$y_3$] (y3) at (1.8,2) {};
\node[vertex,label=right:$y_4$] (y4) at (1.8,3) {};

\draw[->] (x1)--(x2)--(x3)--(x4);
\draw[->] (y1)--(y2)--(y3)--(y4);

\draw[->] (x1)--(y3);
\draw[->] (x2)--(y1);
\draw[->] (x3)--(y4);
\draw[->] (x4)--(y2);

\node at (0.9,-0.9) {$F_{=,4}(\pi)$};
\end{scope}

\begin{scope}[xshift=4.0cm]
\node[vertex,label=left:$x_1$] (a1) at (0,0) {};
\node[vertex,label=left:$x_2$] (a2) at (0,1) {};
\node[vertex,label=left:$x_3$] (a3) at (0,2) {};
\node[vertex,label=left:$x_4$] (a4) at (0,3) {};

\node[vertex,label=right:$y_1$] (b1) at (1.8,0) {};
\node[vertex,label=right:$y_2$] (b2) at (1.8,1) {};
\node[vertex,label=right:$y_3$] (b3) at (1.8,2) {};
\node[vertex,label=right:$y_4$] (b4) at (1.8,3) {};

\draw[->] (a1)--(a2)--(a3)--(a4);
\draw[->] (b1)--(b2)--(b3)--(b4);

\draw[->] (a1)--(b1);
\draw[->] (a1)--(b2);
\draw[->] (a1)--(b3);
\draw[->] (a1)--(b4);

\draw[->] (a2)--(b1);
\draw[->] (a2)--(b2);
\draw[->] (a2)--(b4);

\draw[->] (a3)--(b2);
\draw[->] (a3)--(b4);

\draw[->] (a4)--(b2);

\node at (0.9,-0.9) {$F_{\leq,4}(\pi)$};
\end{scope}

\begin{scope}[xshift=8.0cm]
\node[vertex,label=left:$x_1$] (c1) at (0,0) {};
\node[vertex,label=left:$x_2$] (c2) at (0,1) {};
\node[vertex,label=left:$x_3$] (c3) at (0,2) {};
\node[vertex,label=left:$x_4$] (c4) at (0,3) {};

\node[vertex,label=right:$y_1$] (d1) at (1.8,0) {};
\node[vertex,label=right:$y_2$] (d2) at (1.8,1) {};
\node[vertex,label=right:$y_3$] (d3) at (1.8,2) {};
\node[vertex,label=right:$y_4$] (d4) at (1.8,3) {};

\draw[->] (c1)--(c2)--(c3)--(c4);
\draw[->] (d1)--(d2)--(d3)--(d4);

\draw[->] (c1)--(d3);

\draw[->] (c2)--(d1);
\draw[->] (c2)--(d3);

\draw[->] (c3)--(d1);
\draw[->] (c3)--(d3);
\draw[->] (c3)--(d4);

\draw[->] (c4)--(d1);
\draw[->] (c4)--(d2);
\draw[->] (c4)--(d3);
\draw[->] (c4)--(d4);

\node at (0.9,-0.9) {$F_{\geq,4}(\pi)$};
\end{scope}

\node at (4.9,-1.7)
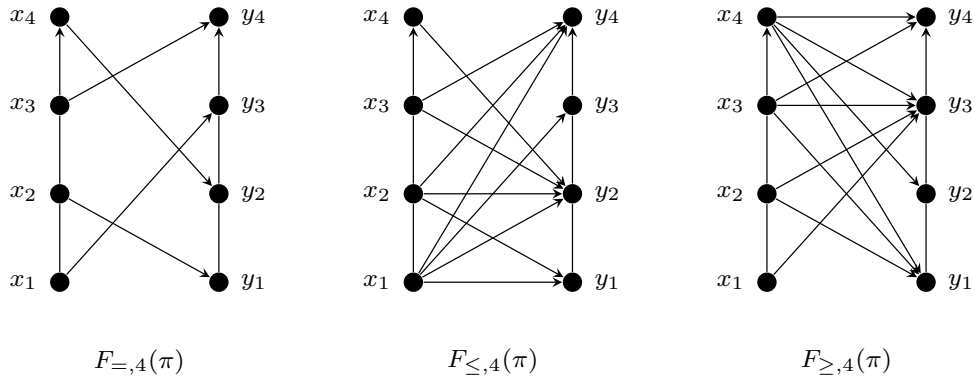
{\small Figure 1. The graphs $F_{Q,4}(\pi)$ for $Q\in\{=,\leq,\geq\}$ and $\pi=(3,1,4,2)$.};

\end{tikzpicture}%
}
\end{center}

For fixed $Q$, the tournaments $F_{Q,n}(\pi)$, over all $n\ge1$ and $\pi\in\mathfrak S_n$, form one of the three canonical families encoding permutations from~\cite{GT}. Throughout this paper, we call their members \emph{permutation-encoding tournaments}. Let $\cF_Q$ consist of all induced subtournaments of these tournaments.

The reason for introducing these three families is that each is an obstruction to bounded twin-width, as shown by Geniet and Thomass\'e.

\begin{lemma}[Geniet--Thomass\'e, Theorem~1.6 in~\cite{GT}]\label{lem:obstruction}
For each $Q\in\{=,\le,\ge\}$, the hereditary tournament class $\cF_Q$ has unbounded twin-width.
\end{lemma}

\section{Proof of Theorem~\ref{thm:transfer}}

\begin{lemma}\label{lem:extract}
Fix a regular semigrid scheme. Suppose an $n\times n$ semigrid $(G,\prec)$ of this scheme satisfies $\omega(G)\le r<n$. Then $I_0$ and $U=I_1\cup\cdots\cup I_n$ are independent. Moreover, there is a fixed $Q\in\{=,\le,\ge\}$, depending only on the scheme, such that for every $\pi\in\mathfrak S_n$, the tournament $\Inv(G,\prec)$ contains an induced copy of $F_{Q,n}(\pi)$.
\end{lemma}

\begin{proof}
We denote the vertex set of $G$ by
\[
I_i=v_{i,1}\prec\cdots\prec v_{i,n+1}, \quad 0\le i \le n.
\]
If $I_0$ were a clique, then $|I_0|=n+1>r$, a contradiction. Inside $U$, regularity assigns one edge/non-edge value to each of the four relative positions, i.e., the same row, the same column, and different rows with increasing or decreasing column indices. If any one of these values were ``edge'', the corresponding set below would be a clique of size at least $n$:
\[
I_1,\qquad
\{v_{j,1}:1\le j\le n\},\qquad
\{v_{j,j}:1\le j\le n\},\qquad
\{v_{j,n+1-j}:1\le j\le n\}.
\]
Since $\omega(G)\le r<n$, all four values must be non-edges. Thus both $I_0$ and $U$ are independent. Consequently, only the position of $I_0$ and the cross-relation $R\in\{=,\ne,\le,\ge\}$ remain to be considered, leaving exactly eight possible schemes.

Fix $\pi\in\mathfrak S_n$ and put $t_j:=\pi^{-1}(j)$. Set
\[
a_i:=v_{0,i}, 1\le i\le n+1,\quad
b_j:=v_{j,t_j}, 1\le j\le n,\quad
\text{and}\quad b_j^+:=v_{j,t_j+1}, 1\le j\le n.
\]
Because $I_0$ and $U$ are independent in $G$, the sequences $(a_1,\ldots,a_n)$ and $(b_1,\ldots,b_n)$ are transitive subtournaments in $\Inv(G,\prec)$, both directed forward in their displayed orders. The same holds for $(b_1^+,\ldots,b_n^+)$. By~\eqref{eq:semigrid-cross} and~\eqref{eq:inverse}, the cross-arcs can be written without complement notation as
\begin{equation}\label{eq:cross-arcs}
\begin{array}{lll}
I_0\prec U &: & b_j\to a_i \ \Longleftrightarrow\ i\,R\,t_j,\\[2pt]
U\prec I_0 &: & a_i\to b_j \ \Longleftrightarrow\ i\,R\,t_j.
\end{array}
\end{equation}

First suppose that $I_0\prec U$. If $R$ is $\ne$, the reverse arc $a_i\to b_j$ occurs exactly when $i=t_j$, so the two sequences directly induce $F_{=,n}(\pi)$. If $R$ is $=$, exchanging the roles of the two sequences gives $F_{=,n}(\pi^{-1})$. In the case $R$ is $\le$, use $x_i:=a_{i+1}$ and $y_j:=b_j$ for $1\le i,j\le n$. Then
\[
x_i\to y_j \Longleftrightarrow i+1>t_j
\Longleftrightarrow i\ge t_j,
\]
so we obtain $F_{\ge,n}(\pi)$. In the case $R$ is $\ge$, use $x_i:=a_i$ and $y_j:=b_j^+$ for $1\le i,j\le n$. Then
\[
x_i\to y_j \Longleftrightarrow i<t_j+1
\Longleftrightarrow i\le t_j,
\]
so we obtain $F_{\le,n}(\pi)$. The extra $(n+1)$st position is exactly what makes these two shifts possible.

\begin{center}
\resizebox{1.1\textwidth}{!}{%
\begin{tikzpicture}[
    vertex/.style={circle,fill=black,inner sep=1.6pt},
    redvertex/.style={circle,fill=red,inner sep=1.6pt},
    >=stealth,
    font=\small
]

\begin{scope}[xshift=0cm]
\node[vertex] (a1) at (0,0) {};
\node[vertex] (a2) at (0.45,0) {};
\node[vertex] (a3) at (0.90,0) {};
\node[vertex] (a4) at (1.35,0) {};
\node[vertex] (a5) at (1.80,0) {};
\node[vertex] (a6) at (2.25,0) {};
\draw[->] (a1)--(a2)--(a3)--(a4)--(a5)--(a6);
\node[font=\fontsize{5}{5}\selectfont] at (0,-0.24) {$a_1$};
\node[font=\fontsize{5}{5}\selectfont] at (0.90,-0.24) {$a_j$};
\node[font=\fontsize{5}{5}\selectfont] at (1.80,-0.24) {$a_n$};
\draw[decorate,decoration={brace,mirror,amplitude=3pt}]
(-0.08,-0.48) -- (2.33,-0.48)
node[midway,below=4pt] {$I_0$};
\end{scope}

\begin{scope}[xshift=3.0cm]
\node[vertex] (b1) at (0,0) {};
\node[redvertex] (b2) at (0.45,0) {};
\node[vertex] (b3) at (0.90,0) {};
\node[vertex] (b4) at (1.35,0) {};
\node[vertex] (b5) at (1.80,0) {};
\node[vertex] (b6) at (2.25,0) {};
\draw[->] (b1)--(b2)--(b3)--(b4)--(b5)--(b6);
\node[font=\fontsize{5}{5}\selectfont] at (0.45,-0.24) {$b_1^+\!:\!t_1\!+\!1$};
\draw[decorate,decoration={brace,mirror,amplitude=3pt}]
(-0.08,-0.48) -- (2.33,-0.48)
node[midway,below=4pt] {$I_1$};
\end{scope}

\node at (6.0,0) {$\cdots$};

\begin{scope}[xshift=6.75cm]
\node[vertex] (c1) at (0,0) {};
\node[vertex] (c2) at (0.45,0) {};
\node[vertex] (c3) at (0.90,0) {};
\node[redvertex] (c4) at (1.35,0) {};
\node[vertex] (c5) at (1.80,0) {};
\node[vertex] (c6) at (2.25,0) {};
\draw[->] (c1)--(c2)--(c3)--(c4)--(c5)--(c6);
\node[font=\fontsize{5}{5}\selectfont] at (1.35,-0.24) {$b_j^+\!:\!t_j\!+\!1$};
\draw[decorate,decoration={brace,mirror,amplitude=3pt}]
(-0.08,-0.48) -- (2.33,-0.48)
node[midway,below=4pt] {$I_j$};
\end{scope}

\node at (9.75,0) {$\cdots$};

\begin{scope}[xshift=10.5cm]
\node[vertex] (d1) at (0,0) {};
\node[vertex] (d2) at (0.45,0) {};
\node[vertex] (d3) at (0.90,0) {};
\node[vertex] (d4) at (1.35,0) {};
\node[vertex] (d5) at (1.80,0) {};
\node[redvertex] (d6) at (2.25,0) {};
\draw[->] (d1)--(d2)--(d3)--(d4)--(d5)--(d6);
\node[font=\fontsize{5}{5}\selectfont] at (2.25,-0.24) {$b_n^+\!:\!t_n\!+\!1$};
\draw[decorate,decoration={brace,mirror,amplitude=3pt}]
(-0.08,-0.48) -- (2.33,-0.48)
node[midway,below=4pt] {$I_n$};
\end{scope}

\draw[->,red,bend left=32] (a1) to (b2);
\draw[->,red,bend left=25] (a3) to (c4);
\draw[->,red,bend left=15] (a5) to (d6);

\node at (6.25,-1.6)
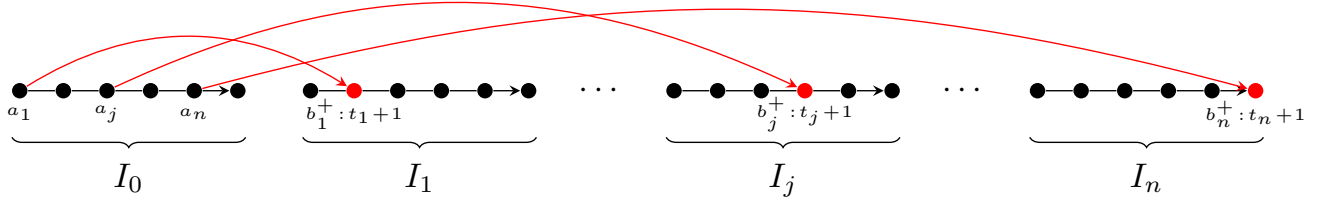
{Figure 2. Example case of $I_0\prec U$ and $R=\ge$, with $b_j^+=v_{j,t_j+1}$, yielding $F_{\le,n}(\pi)$.};

\end{tikzpicture}%
}
\end{center}

Secondly suppose that $U\prec I_0$. If $R\in\{=,\ge,\le\}$,~\eqref{eq:cross-arcs} and~\eqref{eq:permutation-encoding} give $F_{=,n}(\pi)$, $F_{\ge,n}(\pi)$, or $F_{\le,n}(\pi)$ directly. If $R$ is $\ne$, exchanging the roles of the two sequences gives $F_{=,n}(\pi^{-1})$. The eight cases are summarized in Table~\ref{tab:extraction}.

\begin{table}[H]
\centering
\small
\renewcommand{\arraystretch}{1.12}
\begin{tabular}{ccccc}
\toprule
Position & $=$ & $\ne$ & $\le$ & $\ge$ \\
\midrule
$I_0\prec U$ & $F_{=,n}(\pi^{-1})$ & $F_{=,n}(\pi)$ & $F_{\ge,n}(\pi)$ & $F_{\le,n}(\pi)$ \\
$U\prec I_0$ & $F_{=,n}(\pi)$ & $F_{=,n}(\pi^{-1})$ & $F_{\le,n}(\pi)$ & $F_{\ge,n}(\pi)$ \\
\bottomrule
\end{tabular}
\caption{The permutation-encoding tournament extracted in each of the eight remaining cases.}
\label{tab:extraction}
\end{table}
Note that in the two cases where the construction gives $F_{=,n}(\pi^{-1})$, we use the fact that inversion is a bijection of $\mathfrak S_n$. Thus, as $\pi$ ranges over $\mathfrak S_n$, the tournament $\Inv(G,\prec)$ contains an induced copy of $F_{Q,n}(\pi)$ for every $\pi\in\mathfrak S_n$, proving the claim.
\end{proof}

\begin{proof}[Proof of Theorem~\ref{thm:transfer}]
Let
\[
\cC_{k,r}:=\{(T^{\prec},\prec):\tww(T)\le k,\ \omega(T^{\prec})\le r\}.
\]
It is clear that this class is hereditary. Assume for contradiction that $\cC_{k,r}$ has unbounded twin-width. By Lemma~\ref{lem:semigrid}, it contains arbitrarily large regular square semigrids. Since there are only $256$ schemes, one scheme occurs for arbitrarily large values of $n$. A regular semigrid of a fixed scheme contains every smaller regular semigrid of that scheme as an induced ordered subgraph (see~\cite[Section~4]{ST}); hence heredity implies that $\cC_{k,r}$ contains the regular $n\times n$ semigrid of this fixed scheme for every $n$.

Fix $n>r$ and let $(G,\prec)\in\cC_{k,r}$ be the regular $n\times n$ semigrid of this scheme. Lemma~\ref{lem:extract} gives a fixed $Q\in\{=,\le,\ge\}$, depending only on the scheme, such that $\Inv(G,\prec)$ contains an induced copy of $F_{Q,n}(\pi)$ for every $\pi\in\mathfrak S_n$. Since $(G,\prec)\in\cC_{k,r}$, we have $\tww(\Inv(G,\prec))\le k$ and hence
\[
\tww(F_{Q,n}(\pi))\le k.
\]
This holds for every $n>r$ and every $\pi\in\mathfrak S_n$. So all sufficiently large generators of $\cF_Q$ have twin-width at most $k$, which gives that $\cF_Q$ has bounded twin-width and contradicts Lemma~\ref{lem:obstruction}. Therefore $\cC_{k,r}$ has bounded twin-width.
\end{proof}

\begin{proof}[Proof of Corollary~\ref{cor:main}]
Fix $T$ with $\tww(T)\le k$, put $r:=\domega(T)$, and choose an order $\prec$ with
\[
\omega(T^{\prec})=r.
\]
By Theorem~\ref{thm:transfer}, we have $\tww(T^{\prec},\prec)\le h(k,r)$ and thus
\[\tww(T^{\prec})\le h(k,r).\] Lemma~\ref{lem:chi} therefore gives $\chi(T^{\prec})\le p_{h(k,r)}(r)$ and thus
\[
\dchi(T)\le \chi(T^{\prec})\le p_{h(k,r)}(r).
\]
\end{proof}

\section{Concluding remarks}
Theorem~\ref{thm:transfer} is the core result of this paper: it gives a transfer bound for every order whose backward graph has bounded clique number. Corollary~\ref{cor:main} confirms Conjecture 3.13 of Aboulker, Aubian, Charbit, and Lopes by showing that the class of tournaments of bounded twin-width is $\dchi$-bounded. 

After completing this manuscript, we became aware of two other proofs. First, Theorem~\ref{thm:transfer} has a following short proof. By Geniet and Thomass\'e~\cite{GT}, there are at most $c_k^n$ isomorphism types of $n$ vertex tournaments of twin-width at most $k$. By Xing~\cite{Xing}, for every $r$, each $n$ vertex tournament admits at most $d_r^n$ orderings whose backward graph is $K_{r+1}$-free. Consequently, for fixed $k$ and $r$, the hereditary class of ordered graphs $(T^{\prec},\prec)$ satisfying $\tww(T)\le k$ and $\omega(T^{\prec})\le r$ has single exponential growth. The growth characterization of Simon and Toru\'nczyk~\cite{ST} then shows that this class has bounded twin-width, proving Theorem~\ref{thm:transfer}. Secondly, Aubian and Coulomb~\cite{AC} independently proved(or noted, as they would said) Corollary~\ref{cor:main}. Their proof uses the fact given in \cite{GT} that every bounded-twin-width class of tournaments excludes a fixed tournament whose backedge graph is a matching. Moreover, Briański, Davies, and Walczak~\cite{BDW} proved that ordered graphs excluding a fixed ordered matching as an induced subgraph are $\chi$-bounded. Together, these imply the $\vec\chi$-boundedness.

Our theorem does not resolve the stronger Conjecture 3.14 from~\cite{AACL}, since our bound on $\tww(T^{\prec},\prec)$ depends on both $\tww(T)$ and the clique number of the chosen backward graph. Two natural next problems are to obtain polynomial bounds for $h(k,r)$, and to seek an algorithmic version that produces useful orderings and dicolorings. It would also be interesting to understand whether the semigrid obstruction argument extends to broader classes of oriented graphs.

\section*{Acknowledgments}
We thank Pierre Aboulker and Samuel Coulomb for pointing out the proof in~\cite{AC} and helpful discussions about this problem.

\section*{Declaration of AI}
ChatGPT-5.6 sol, developed by OpenAI, was used during the early stages of this project and in the preparation of the manuscript. In particular, the core idea underlying the proof of the main theorem was proposed by ChatGPT.
ChatGPT was also used to assist with the language, organization, and presentation of the manuscript. The authors subsequently developed and independently verified all mathematical arguments and take full responsibility for all statements, proofs, citations, and conclusions presented in this paper.


\begin{thebibliography}{99}

\bibitem{NeumannLara}
V. Neumann-Lara,
\emph{The dichromatic number of a digraph},
J. Combin. Theory Ser. B 33 (1982), 265--270.

\bibitem{BCCHFLSST}
E. Berger, K. Choromanski, M. Chudnovsky, J. Fox, M. Loebl,
A. Scott, P. Seymour, and S. Thomass\'e,
\emph{Tournaments and colouring},
J. Combin. Theory Ser. B 103 (2013), 1--20.

\bibitem{CSSS}
M. Chudnovsky, A. Scott, P. Seymour, and S. Spirkl,
\emph{Pure pairs. X. Tournaments and the strong Erd\H{o}s--Hajnal property},
European J. Combin. 115 (2024), 103786.

\bibitem{AACL}
P. Aboulker, G. Aubian, P. Charbit, and R. Lopes,
\emph{Clique number of tournaments},
arXiv:2310.04265v2, 2026.

\bibitem{Kim}
I. Kim,
\emph{On containment relations in directed graphs},
Ph.D. thesis, Princeton University, 2013.

\bibitem{NSS}
T. Nguyen, A. Scott, and P. Seymour,
\emph{Some results and problems on tournament structure},
Journal of Combinatorial Theory, Series B 173 (2025), 146--183.

\bibitem{ST}
P. Simon and S. Toru\'nczyk,
\emph{Ordered graphs of bounded twin-width},
arXiv:2102.06881, 2021.

\bibitem{GT}
C. Geniet and S. Thomass\'e,
\emph{First Order Logic and Twin-Width in Tournaments and Dense Oriented Graphs},
European Journal of Combinatorics 132 (2026), 104247;
arXiv:2207.07683v5.

\bibitem{BT}
R. Bourneuf and S. Thomass\'e,
\emph{Bounded twin-width graphs are polynomially $\chi$-bounded},
Advances in Combinatorics 2025:2 (2025), 19 pp.

\bibitem{BCTT}
R. Bourneuf, J. Cocquet, C. Tang, and S. Thomass\'e,
\emph{A polynomial-time approximation algorithm for complete interval minors},
in \emph{Approximation, Randomization, and Combinatorial Optimization. Algorithms and Techniques (APPROX/RANDOM 2025)},
Leibniz International Proceedings in Informatics 353 (2025), 15:1--15:23.

\bibitem{BKTW}
\'E. Bonnet, E. J. Kim, S. Thomass\'e, and R. Watrigant,
\emph{Twin-width I: tractable FO model checking},
Journal of the ACM 69(1) (2022), Article 3.

\bibitem{BGKTW}
\'E. Bonnet, C. Geniet, E. J. Kim, S. Thomass\'e, and R. Watrigant,
\emph{Twin-width III: Max Independent Set, Min Dominating Set, and Coloring},
SIAM Journal on Computing 53(5) (2024), 1602--1640.

\bibitem{PS}
M. Pilipczuk and M. Soko{\l}owski,
\emph{Graphs of bounded twin-width are quasi-polynomially $\chi$-bounded},
Journal of Combinatorial Theory, Series B 161 (2023), 382--406.

\bibitem{Xing}
Y. Xing,
\emph{Excluded structures for pinch-graphic matroids and tournaments},
Master's thesis, University of Waterloo, 2026.

\bibitem{BDW}
M. Briański, J. Davies, and B. Walczak,
\emph{Colouring graphs without an induced ordered matching},
In preparation.

\bibitem{AC}
G. Aubian and S. Coulomb,
\emph{Clique number of tournaments II},
to appear.

\end{thebibliography}
\end{document}